\documentclass[a4paper,11pt]{article}
\usepackage{multirow}%
\usepackage{xcolor}%
\usepackage{fullpage}
\usepackage{amsmath,amsthm,amsfonts,amssymb,euscript}%
\usepackage[dvips]{graphicx}
\usepackage{colortbl}
\usepackage{enumerate}
\newcommand{\mathds}[1]{#1}

\usepackage{titletoc}
\usepackage[runin]{abstract}
\usepackage{calc}

\newtheorem{theorem}{Theorem}
\newtheorem{corollary}[theorem]{Corollary}

\newtheorem{lemma}[theorem]{Lemma}

\newtheorem{remark}{Remark}

\newcommand{\E}{\mathbb{E}}
\renewcommand{\P}{\mathbb{P}}
\newcommand{\R}{\mathbb{R}}
\newcommand{\N}{\mathbb{N}}

\newcommand{\w }[1]{\widehat{#1}}

\newcommand{\dd}{ \mathrm{d} }
\renewcommand{\r}{ \rightarrow }

\newcommand{\vertiii}[1]{{\left\vert\kern-0.25ex\left\vert\kern-0.25ex\left\vert #1 
    \right\vert\kern-0.25ex\right\vert\kern-0.25ex\right\vert}}

\newcommand{\lr}{ \longrightarrow }

\DeclareMathOperator{\spann}{span} 
\DeclareMathOperator{\var}{var}

\newcommand{\cRM}[1]{\MakeUppercase{\romannumeral #1}}

\renewcommand{\t}[1]{\text{#1}}
\newcolumntype{C}[1]{>{\centering\arraybackslash }b{#1}}

\begin{document}

\title{On the acceleration of some empirical means with application to nonparametric regression}

\author{Bernard Delyon and Fran\c cois Portier}
\date{}
\maketitle

   {\scshape Abstract:}\ Let $(X_1,\ldots ,X_n)$ be an i.i.d. sequence of random variables in $\R^d$, $d\geq 1$, for some function $\varphi:\R^d\r \R$, under regularity conditions, we show that
\begin{align*}
 n^{1/2} \left(n^{-1} \sum_{i=1}^n \frac{\varphi(X_i)}{\w f^{(i)}(X_i)}-\int_{} \varphi(x)dx \right) \overset{\P}{\lr} 0,
\end{align*}
where $\w f^{(i)}$ is the classical leave-one-out kernel estimator of
the density of $X_1$. This result is striking because it speeds up 
traditional rates, in root $n$, derived from the central limit theorem 
when $\w f^{(i)}=f$. As a consequence, it improves the classical 
Monte Carlo procedure for integral approximation. The paper mainly addressed with theoretical issues related to the later result (rates 
of convergence, bandwidth choice, regularity of $\varphi$) but also 
interests some statistical applications dealing with random design regression. In particular, 
we provide the asymptotic normality of the estimation of the linear functionals of a regression function on which the only requirement is 
the  H\"older regularity. This leads us to a new version of the \textit{average derivative estimator} introduced by H\"ardle and Stoker in \cite{hardle1989} which allows for \textit{dimension reduction} by estimating the \textit{index space} of a regression.

\bigskip


\bigskip

\section{Introduction}

Let $(X_1,\ldots ,X_n)$ be an i.i.d. sequence of random variables in $\R^d$, $d\geq 1$, for some function $\varphi:\R^d\r \R$, under regularity conditions, we show that
\begin{align}\label{the result}
 n^{1/2} \left(n^{-1} \sum_{i=1}^n \frac{\varphi(X_i)}{\w f^{(i)}(X_i)}-\int_{} \varphi(x)dx \right) \overset{\P}{\lr} 0,
\end{align}
where $\w f^{(i)}$ is the classical leave-one-out kernel estimator of the density of $X_1$ say $f$, defined by
\begin{align*}
\w f^{(i)}(x)  = (nh^d)^{-1}\sum_{j\neq i}^n K(h^{-1} (X_j-x)),\qquad \text{for every }x\in \R^d,
\end{align*}
where $K$ is a $d$-dimensional kernel and where $h$, called the bandwidth, needs to be chosen and will certainly depend on $n$.
 Result (\ref{the result}) and the central limit theorem lead to the following reasoning: when estimating the integral of a function that is evaluated on a random grid $(X_i)$, whether $f$ is known or not, using a kernel estimator of $f$ provides better convergence rates than using $f$ itself. A first obvious application of this result is for Monte Carlo integration
 when the design, i.e. the distribution of the points, is not controlled. If the design 
 is free, other methods exist, like quasi random numbers, which may prove to be 
 more efficient, depending on the regularity of the function and on the dimension 
 (we refer to \cite{caflisch1998} for a comprehensive presentation of these methods). In this paper, we are interested in the random design case for which the previous methods as Quasi Monte Carlo and grid integration cannot be implemented.

Equation (\ref{the result}) may have applications in nonparametric regression with random design. Let
\begin{align}\label{modelad}
Y_i= g(X_i) + \sigma(X_i) e_i,
\end{align}
where $(e_i)$  is an i.i.d. sequence of real random variables independent of the sequence $(X_i)$, and $\sigma:\R^d \r \R$ and $g:\R^d \r \R$ are unknown functions. In this context, one of the most evident use of Equation (\ref{the result}) deals with the estimation of the linear functionals of $g$, i.e. the quantities $\int g(x)\psi(x)dx$ for some functions $\psi:\R^d\r \R$. Under regularity conditions, we show that
\begin{align}\label{theresult2}
 \ n^{1/2} \left(n^{-1} \sum_{i=1}^n \frac{Y_i \psi(X_i)}{\w f^{(i)}(X_i)}-\int_{} g(x)\psi(x)dx \right)\overset{\dd}{\lr} \mathcal N (0,v),
\end{align}
where $v=\var((Y-g(X_1))\psi(X_1)f(X_1)^{-1})$. Among typical applications of Result (\ref{theresult2}), we can mention Fourier coefficients estimation for either nonparametric estimation (see for instance \cite{hardle1990}, section 3.3) or location parameter estimation (see \cite{gamboa2007} and the reference therein). We shall focus on applications dealing with the \textit{multiple index model}, i.e. when the link function $g(x)=g_0(\beta^T x)$ for every $x\in \R^d$, with $\beta\in \R^{d\times p}$ called the index, $p\leq d$. As it was noticed by H\"ardle and Stoker in \cite{hardle1989}, for the \textit{average derivative estimator} (ADE), when $\psi=\nabla f$ the estimator in Equation (\ref{theresult2}) recovers the \textit{index} with rates root $n$. Their method is popular, notably because it is a direct estimation procedure that does not involve complicated optimization algorithm. Thanks to Result \ref{theresult2}, we shall see that choosing different functions $\psi$ than $\nabla f$ may lead to an accurate estimation of the \textit{index space} $\spann(\beta)$.

 The estimation of the linear functionals of $g$ is a typical semiparametric problem in the sense that it requires the nonparametric estimation of $f$ as a first step and then to use it in order to estimate a real parameter. To the best of our knowledge, estimators that achieve root $n$ consistency have not been provided yet in the case of a regression with random design. Our approach is based on kernel estimates $\w f^{(i)}$ of the density of $X_1$ that is then plugged into the classical empirical estimator of the quantity $\E[Y\psi(X) f(X)^{-1}]$. There is at least four main interesting facts about the weak convergence (\ref{theresult2}). They are listed below.

\begin{enumerate}[(A)]
\item \label{p1} The first point about Equation (\ref{theresult2}) is that, despite slower rates than root $n$ obtained when estimating $f$, the final estimator recovers the parametric rate root $n$. Similar facts have already been noticed by some authors in different semiparametric problems as, among others, by Stone in \cite{stone1975} in the case of the estimation of a location parameter, by Robinson in \cite{robinson1988} in a \textit{partially linear regression model}, or by H\"ardle and Stoker in \cite{hardle1989} studying ADE (see also \cite{ichimura1993} and \cite{patilea2003} about the semiparametric $M$-estimation). 

\item \label{p2}Going further in the analysis of Result (\ref{theresult2}), we notice that the asymptotic variance $v$ is smaller than the asymptotic variance of the estimator with the true density (see Equation (\ref{varcomp}) in Remark \ref{rem5bis}). As a consequence for this problem, there is an asymptotic gain in estimating the density. We might remark that the underlying cause is Result (\ref{the result}) because it implies that the asymptotic variance $v$ stems only from the noise $e_i$ associated to the observation of $Y_i$ in Model (\ref{modelad}). Surprisingly there is not any terms in $v$ that are due to the randomness of the design. 

\item \label{p3} Despite similarities between our estimator and some estimators of the semiparametric literature (e.g. the references in Point (\ref{p1})), the technical details of our approach are different since they are based on Equation (\ref{the result}). A similar result was originally stated by Vial in \cite{vial2003} (Chapter 7, Equation (7.27)) in the \textit{multiple index model} context. 

\item \label{p4} Unfortunately, it turns out that Result (\ref{the result}) is no longer true when estimating functionals of the form $f\mapsto \int  T(x,f(x)) dx$ where $T:\R^2\rightarrow \R$ is different from the map $(x,y)\mapsto \varphi(x)$ (see Section \ref{generalizing}). As a result, it suggests that Point (\ref{p2}) has no reason to hold when estimating $\int g(x) T(x,f(x)) dx $ with our approach. In view of the asymptotic variance of ADE expressed in Equation (\ref{varade}), this kind of suboptimal properties happen also for ADE where the transformation $T$ differs from the map $(x,y)\mapsto \varphi(x)$ and involves the derivative of $f$. As a consequence, it might be better to replace, in ADE, the derivatives of $f$ by the derivatives of a known function.

\end{enumerate}


The paper is organized as follows. Section \ref{s1} deals with technical issues related to Equation (\ref{the result}). In particular, we examine the rates of convergence of (\ref{the result}) according to the choice of the bandwidth, the dimension and the regularity of the functions $\varphi$ and $f$. We also introduce a corrected estimator that converges to $0$ faster than the initial one given in Equation (\ref{the result}). This corrected estimator allows a less restrictive choice of the bandwidth. Section \ref{s2} is dedicated to the estimation of the linear functionals of $g$. We show Result (\ref{theresult2}) under mild conditions on $g$ that only needs to be piecewise H\"older. In Section \ref{s3}, we focus on the application of our results in the context of the \textit{multiple index model}. We provide a new version of ADE that might be more efficient (see point (\ref{p4}). We give some simulations that compare our method with ADE and \textit{inverse regression methods} introduced by Li in \cite{li1991} that typically ask more than ADE on the distribution of $X$.

\section{Integral approximation by kernel smoothing}\label{s1}

Let $Q\subset \R^d$ be the support of $\varphi$. The quantity $I(\varphi)= \int  \varphi(x)dx $ is estimated by 
\begin{align*}
\w I(\varphi) = n^{-1} \sum_{i=1}^n \frac{\varphi(X_i)}{\w f^{(i)}(X_i)} 
\end{align*}
We define the leave-one-out estimator of the variance of $h^{-p}K(h^{-1}(x-X_j))$ by  
\begin{align*}
\w v^{(i)}(x) = ((n-1)(n-2))^{-1} \sum_{j\neq i}^n (h^{-d}K(h^{-1}(x-X_j))- \w f^{(i)}(x))^2 ,
\end{align*}
this one is needed to correct the initial estimator by
\begin{align*}
\w I_{cor}(\varphi) = n^{-1} \sum_{i=1}^n \frac{\varphi(X_i)}{\w f^{(i)}(X_i)}\left(1-\frac{\w v^{(i)}(X_i)}{\w f^{(i)}(X_i)^2}\right) .
\end{align*}
To state our main result about the convergences of $\w I(\varphi)$ and 
$\w I_{cor}(\varphi) $, we define the Nikol'ski class $\EuScript H_s$ 
of functions of regularity $s=k+\alpha$,
$k\in \mathbb N$, $0<\alpha\le 1$ as the set of $k$ times differentiable functions
$\varphi$ such that all its derivatives of order $k$ satisfy
\cite{tsybakov2009}
\begin{align}\label{nicol}
\int (\varphi^{(l)}(x+u)-\varphi^{(l)}(x))^2dx\le C|u|^{2\alpha},~~l=(l_1,\dots , l_d), ~~\sum l_i\le k.
\end{align}
Be careful that $k=\lfloor s\rfloor$, with the convention that $\lfloor n\rfloor=n-1$ if $n\in\mathbb N$. We need the following assumptions.

\begin{enumerate}[(\text{A}1)]
\item \label{ash1} For some ${s}>0$ the function $\varphi$ belongs to $\EuScript H_{s}$ on 
$\R^d$ and has compact support $Q$.
\item \label{ash2} The variable $X_1$ has a bounded density $f$ on $\R^d$ such that its $r$-th 
order derivatives are bounded.
\item \label{ash3} For every $x\in Q$, $f(x)\geq b>0$.
\item \label{ash4} The kernel $K$ is symmetric with order $r\ge{s}$.
Moreover, for every $x\in \R^d$, $K(x)\leq C_1\exp(-C_2\|x\|)$ for some constants $C_1$ and $C_2$.
\end{enumerate}

The next theorem is proved in the appendix.

\begin{theorem}\label{thelemma}
Assume that \emph{(A\ref{ash1}-A\ref{ash4})} hold, we have the following $O_\P$ estimates
\begin{align}\label{borne1}
& n^{1/2} \left(\w I_{}(\varphi) -\int \varphi(x) dx \right) 
 = O_\P\left( h^{s} + n^{1/2}h^{r} +  n^{-1/2}h^{-d} \right),\\
 &n^{1/2} \left(\w I_{cor}(\varphi)
 -\int \varphi(x)  dx \right)= O_\P\left( h^{s} + n^{1/2}h^{r} +  n^{-1/2}h^{-d/2} +  n^{-1} h^{-3d/2} \right)\label{borne2}
\end{align}
which are valid if the sums inside the $O_\P$'s tend to zero.
\end{theorem}

\begin{remark}\label{rem1}\normalfont

Assumption (A\ref{ash2}) about the smoothness of $f$ is crucial to guarantee the convergences stated in Theorem \ref{thelemma}. On the one hand, $r$ needs to be greater than $d$ to obtain convergence (\ref{borne1}), on the other hand, $r$ greater than $3d/4$ suffices to get convergence (\ref{borne2}). In the case where each previous assumption fails, there does not exist $h$ such that Equation (\ref{borne1}) or Equation (\ref{borne2}) hold. This phenomenon is often referred as the \textit{curse of dimensionality}. The choice of the bandwidth can be made regarding the $O_{\P}$ estimates in Theorem \ref{thelemma} and assuming that $h=Cn^{-a}$.  To select the parameter $a$, one can optimize the quantity in the $O_{\P}$ in order to derive the best possible rate of convergence. For instance, assuming that $r$ and $s$ are sufficiently large so that the first terms in equations (\ref{borne1}) and (\ref{borne2}) and the last term in Equation (\ref{borne2}) are negligible ($2r>3d$ and $2s>r-d/2$), we obtain the optimal rates $n^{-\frac{(r-d)}{2(r+d)}}$ and $n^{-\frac{(r-d/2)}{2(r+d/2)}}$ for bandwidth $h\propto n^{-\frac{1}{r+d}}$ and $h\propto n^{-\frac{1}{r+d/2}}$, respectively. As in the semiparametric problem studied in \cite{hardle1989} (see section 4.1), our estimator of $f$ is suboptimal with respect to the density estimation problem (see \cite{stone1980}). Indeed, to achieve the optimal rates of density estimation one needs to have $h \propto n^{-1/(2r+d)}$ which contradicts the fact that the bias goes to $0$ in Theorem \ref{thelemma}. However the choice of the constant $C$ in the bandwidth is not studied here. One can follow H\"ardle, Hart, Marron and Tsybakov (1992) and optimized an equivalent of the MSE, in order to obtain $C$. 
\end{remark}

\begin{remark}\label{rem1bis}\normalfont
Assumption (A\ref{ash2}) neglects the bias problems in the estimation of $f$ that may occur at the borders of $Q$. Indeed, if $f$ has a jump on the boundary of $Q$, then our estimate of $f$ would be asymptotically biased and the rates provided by Theorem \ref{thelemma} does not hold. To get ride of this problem, one can correct by hand the estimator, as for instance in \cite{jones1993}, or use Beta kernels as detailed in \cite{chen1999}. The simulations provided in Figure \ref{fig0} highlight how this problem affects the estimation by considering two different densities.
\end{remark}

\begin{figure}\centering
\includegraphics[width=13.6cm,height=8cm]{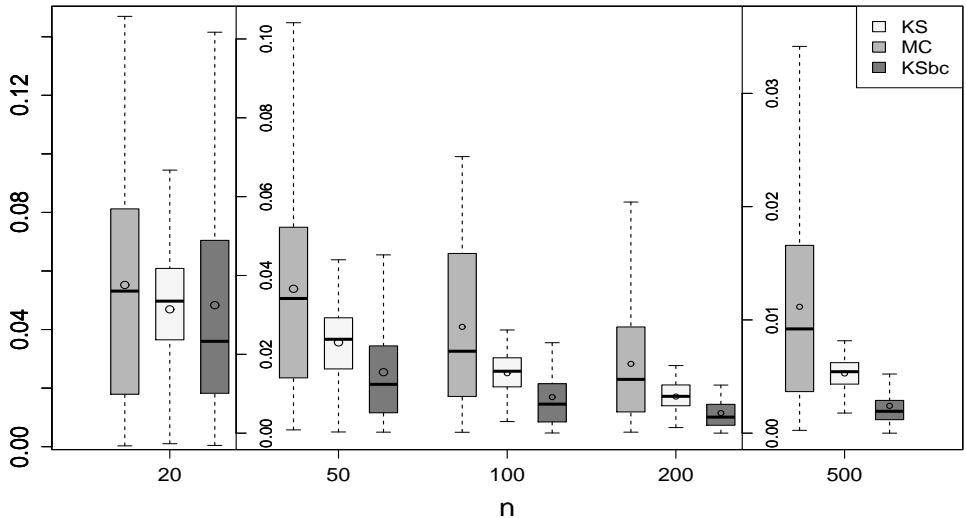} 
\caption{Boxplot over $100$ samples of the estimation error of $\int_0^1 \sin(\pi x)dx$, by the classical Monte Carlo procedure with $f=\mathds 1 _{[0,1]}$, noted MC; by the kernel smoothing with $f=\mathds 1 _{[0,1]}$, the Epanechnikov kernel and $h=n^{-1/3}$, noted KS; by the kernel smoothing with bias correction with $f=\mathds 1 _{[-h,1+h]}$ the Epanechnikov kernel and $h=n^{-1/3}$ noted KSbc; for different sample number.}\label{fig0}
\end{figure}

\begin{remark}\label{rem2} \normalfont
Assumption (A\ref{ash3}) basically says that $f$ is separated from $0$ on $Q$. The exponential bound on the kernel in Assumption (A\ref{ash4}) guarantee that $f$ is estimated uniformly on $Q$ (see \cite{devroye1980}). This leads to $(\inf_{x\in Q} \w f(x))^{-1}=O_{\P}(1)$ and helps to control the random denominator $\w f^{(i)}(X_i)$ in the expression of $\w I(\varphi)$ and $\w I_{cor}(\varphi)$. In the context of Monte Carlo procedure for integral approximation, Assumption (A\ref{ash2}) and Assumption (A\ref{ash3}) are not at all restrictive because it is always possible to draw the $X_i$'s from any probability distribution smooth enough and whose support contains the integration domain.
\end{remark}

\begin{remark}\label{rem3}\normalfont
The use of leave-one-out estimators in $\w I(\varphi)$ and $\w I_{cor}(\varphi)$ is not only justified by the simplification they involve in the proof (some diagonal terms disappear from the sums). It also leads to better convergence rates. For instance, let us consider the term $\w R_0$ in the proof of Equation (\ref{borne2}) in Theorem \ref{thelemma}. Replacing the leave-one-out estimator of $f$ by the classical one, $\w R_0$ remains a degenerate U-statistic but with nonzero diagonal terms. It is easy to verify that those terms lead to the rates $n^{-1/2}h^{-d})$ which is greater than the rate we found for $\w I_{cor}(\varphi)$.

\end{remark}

\begin{remark}\label{rem4}\normalfont

The function class $\EuScript H_s$ contains two interesting sets of 
functions that provide different rates of convergence in Theorem \ref{thelemma}. First, if $\varphi$ is 
$\alpha $-H\"older on $\R^p$ with bounded support, then $\varphi$ 
belongs to $\EuScript H_\alpha$. Secondly, if the support of $\varphi$ is a bounded convex set and 
$\varphi$ is $\alpha$-H\"older inside its support 
(e.g. the indicator of a ball)
then $\varphi\in\EuScript H_{\min(\alpha,1/2)}$ (see 
Theorem~\ref{holderremark} in the appendix). As a result, this loss of 
smoothness at the boundary of the support involves a loss in the rates 
of convergence (\ref{borne1}) and (\ref{borne2}). Precisely, 
whatever the smoothness degree of the function inside its support, if 
continuity fails at the boundary, rates are at most in $h^{1/2}$.

\end{remark}

\section{Estimating the linear functionals of a regression function}\label{s2}

Let $Q\subset \R^d$ be a compact set and $L_2(Q)$ be the space of squared-integrable functions on $Q$. We endowed $L_2(Q)$ with the canonical inner product so that it is an Hilbert space. We consider model (\ref{modelad}) assuming that $g\in L_2(Q)$. Let $\psi\in L_2(Q)$ be extended to $\R^d$ by $0$ outside of $Q$ ($\psi$ has compact support $Q$). The inner product in $L_2(Q)$ between the regression function $g$ and $\psi$, is given by
\begin{align*}
c =  \int_{} g(x)\psi(x)dx ,
\end{align*} 
note that if $\psi$ belongs to a given basis of $L_2(Q)$, then $c$ is a coordinate of $g$ inside this basis. We define the estimator
\begin{align*}
\w c  = n^{-1}\sum_{i=1}^n \frac{Y_i \psi(X_i)}{\w f^{(i)}(X_i)},
\end{align*} 
to derive the asymptotic of $\sqrt n (\w c -c )$, we use Model (\ref{modelad}) to get the decomposition
\begin{align}\label{decomp}
\sqrt n (\w c -c )= \w S + \w R,
\end{align}
with
\begin{align*}
\w R&= \ n^{-1/2} \left(\sum_{i=1}^n \frac{g(X_i)\psi(X_i)}{\w f^{(i)}(X_i)}-\int_{} g(x)\psi(x) dx  \right)\\
\w S &= n^{-1/2} \sum_{i=1}^n \frac{\sigma(X_i)\psi(X_i)}{\w f^{(i)}(X_i)}e_i.
\end{align*} 
Under some conditions, Theorem \ref{thelemma} provides that $\w R$ is 
negligible with respect to $\w S$. As a result, $\w S$ carries the weak 
convergence of $\sqrt n (\w c -c )$, and then the limiting distribution 
can
be obtained making full use of the independence between the $X_i$'s and 
the $e_i$'s. In order to follow this program, this assumptions are needed.

\begin{enumerate}[(\text{A}1)]\setcounter{enumi}{4}
\item \label{ash5} The function $\psi$ is H\"older on its support $Q\subset \R^d$ nonempty bounded and convex.
\item \label{ash6} The function $g$ is H\"older on $Q$ and $\sigma$ is bounded. 
\item \label{ash7} The bandwidth verifies $n^{1/2} h^r\r 0$ and $n^{1/2}h^{d}\r 
+\infty$ as $n$ goes to infinity.
\end{enumerate}

The following theorem is proved in the appendix. 

\begin{theorem}\label{thelemma2}
Assume that \emph{(A\ref{ash2}-A\ref{ash7})} hold, we have
\begin{align*}
n^{1/2}(\w c-c) \overset{\dd } {\lr} \mathcal N (0, v ),
\end{align*}
where $v$ is the variance of the random variable $ \frac{Y_1-g(X_1)}{f(X_1)} \psi(X_1)$.
\end{theorem}

\begin{remark}\label{rem5}\normalfont
The set $Q$ reflects the domain where $g$ is studied. Obviously, the more dense the $X_i$'s in $Q$, the more stable the estimation. Nevertheless, it could happened that $f$ vanishes somewhere on $Q$ and this is not taken into account by our framework. In such situations, one may adapt $Q$ from the sample such that the estimated density does not take too small values. This method called \textit{trimming} (employed for instance in \cite{hardle1989}) guarantees computational stability as well as some theoretical properties. Even if such an approach is feasible here, it involves much more technicalities in the proofs and may cause a loss in the clarity of the statements.
\end{remark}

\begin{remark}\label{rem5bis}\normalfont
The nonstandard convergence rates observed in Theorem \ref{thelemma} 
impacts Theorem \ref{thelemma2} in the following way. Let us compare both estimate $\w c$ and $\widetilde c  = n^{-1}\sum_{i=1}^n Y_i \psi(X_i)f(X_i)^{-1}$ where the latter requires to know $f$. First, if the signal is observed without noise, that is $Y_i=g(X_i)$, then $n^{1/2}(\w c-c) $ goes to $0$ in probability and  $\widetilde c$ is asymptotically normal. Secondly, when there is some noise in the observed signal, that is $e_i\neq 0$, the comparison can be made regarding their asymptotic variances. Since we have
\begin{align}\label{varcomp}
v= \var\left(\frac{Y_1}{f(X_1)} \psi(X_1)\right) -\var\left(\frac{g(X_1)}{f(X_1)} \psi(X_1)\right)\leq \var(n^{1/2}(\widetilde c-c)),
\end{align}
it is asymptotically better to plug the nonparametric estimator of $f$ than to use $f$ directly.

\end{remark}

\section{Applications to multiple index models}\label{s3}

\subsection{Average derivative estimator}

\label{cade}

The multiple index model is defined as Model (\ref{modelad}) with the specification 
\begin{align}\label{index}
g(x)=g_0( \beta^Tx), \qquad \text{for every }x\in \R^{d}	,
\end{align} 
where $\beta\in \R^{d\times p}$, and $p$ is minimal. Under some conditions, essentially 
that $X_1$ has a density \cite{portier2013}, $E=\spann (\beta)$ is unique, it is called the \textit{index space} and the term \textit{index} denotes any of its basis. From now, we assume that $E$ is unique. Our approach is based on the gradient of the regression curve since $\nabla g(x) \in E$.

Under some regularity conditions (see \cite{powell1989}), by the integration by parts formula, we have that
\begin{align}\label{ipp}
\beta_{\psi} = \int  g(x) \nabla\psi(x)dx  =-\int \nabla g(x) \psi(x) dx \in E,
\end{align}
for any smooth function $\psi:\R^d \r \R$. In view of Theorem \ref{thelemma2}, the following estimator
\begin{align}\label{betachap}
\w\beta_{\psi} = n^{-1} \sum_{i=1}^n \frac {Y_i\nabla \psi(X_i)}{\w f^{(i)}(X_i)},
\end{align} 
is root $n$-consistent in estimating a direction of the index space. 
By applying Theorem \ref{thelemma2}, we obtain the following corollary where (A\ref{ash5}) becomes
\begin{enumerate}[(\text{A}1)]\setcounter{enumi}{4}
\item \label{ash5prime} The function $\nabla \psi$ is H\"older on its support $Q\subset \R^d$ nonempty bounded and convex.
\end{enumerate}

\begin{corollary}\label{cor}
Assume that \emph{(A\ref{ash2}-A\ref{ash7})} hold, we have
\begin{align*}
n^{1/2}(\w\beta_{\psi}-\beta_{\psi}) \overset{\dd } {\lr} \mathcal N (0, v ),
\end{align*}
where $v$ is the variance of the random variable $ \frac{Y_1-g(X_1)}{f(X_1)} \nabla \psi(X_1)$.
\end{corollary}
In order to recover the whole space $E$, we have to compute several $\w \beta_{\psi}$, say $(\w \beta_{1},\cdots,\w \beta_{K})$ associated with several functions $\psi=\psi_1,\dots\psi_K$ 
and assume in addition that Equation (\ref{ipp}) holds true for each $\w \beta_k$. The estimate  $\w E$ of $E$
will be taken as the $p$-dimensional space from which the $\w\beta_k$'s are the closest;
there is several ways to do this (PCA, weighted PCA...) and they will be presented in the next section. Note that we assume that the dimension $p$ of $E$ is known, in practice it can be estimated using hypothesis testing \cite{portier2014}.

As the ADE method \cite{hardle1989}, the method we have just described is based on the integration by part formula (\ref{ipp}). As a result, our method may be seen as a new version of ADE, called average derivative estimator by test functions (ADETF). The main difference between ADE and ADETF is that ADE puts $\psi=f$ so that their estimates only recover a single direction. This problem has been circumvented in the recent study \cite{zeng2010} where the authors consider $\psi=\widetilde \psi\nabla f+\nabla\widetilde \psi f$ for some $\widetilde \psi$. First, by considering different functions $ \psi$, our estimator is able to recover the multiple index. Secondly, comparing to both latter references, our approach does not need to estimate the derivatives of the density, and as a result does not require to select two different bandwidths. Moreover the presence of $\nabla \w f$ in ADE may induce an unnecessary noise that could affect badly the estimation. In the asymptotic variance of ADE
\begin{align}\label{varade}
\var\left(\nabla g(X) + \frac{(Y-g(X))\nabla f(X)}{f(X)}\right),
\end{align}
see Theorem 3.1 of \cite{hardle1989}, this is reflected by the additional term $\nabla g(X)$ that does not affect the variance of ADETF provided in Corollary \ref{cor}.

\subsection{Parameter setting}
\label{parset}
\paragraph{Choice of the bandwidth and the kernel.} Theoretical results provided by Corollary \ref{cor} require the use of a high order kernel to reduce the bias. Since our simulations have highlighted that the use of high order kernels are not as crucial in practice as in theory, we consider the Epanechnikov radial kernel given by
\begin{align*}
K(x)\propto (1-\|x\|^2),
\end{align*}
such that $\int K=1$. Contrarily to ADE, it turns out that ADETF is not really affected by the choice of the bandwidth. As a result, in the whole study, we select the optimal bandwidth for ADE and we put $h=2sn^{-1/(d+2)}$ for ADETF, where $s$ is the estimated standard deviation of $X$.
\paragraph{Choice of the test functions.} We define
\begin{align*}
\psi(x) = \widetilde{\psi}(& h_0 ^{-1}\|x\|)\qquad \t{with}\quad \widetilde{\psi}(z)=(1-z)^2(1+z)^2\mathds{1}_{\{|z|<1\}}
\end{align*}
where the scaling parameter $h_0$ is equal to the empirical estimator of $s=\E[\|X-\E[X]\|^2]^{1/2}$,
and our test functions are
\begin{align*}
\psi_k(x) = \psi(x-t_k),~~~k=1,\dots K.
\end{align*}
Observing that better results are obtained if we do not restrict ourselves to
a small value of $K$, we ended up with the simple choice $t_k=X_k$. 

\paragraph{Computation of the directions.} 
We have to extract $p$ directions from $(\w \beta_k)_{k=1,...,n}$. Two  
approaches can be used and combined.
\begin{enumerate}[a)]
\item \label{refresh} Use a criterion of dependence between $Y$ and $\w\beta_k^T X$ to select
among the $\w\beta_k$'s.
\item \label{acp} Choose the best direction through a PCA of these vectors.
\end{enumerate}
The set $(\w \beta_k)_{k=1,...,n}$ is an heterogeneous family of estimated vector. Indeed because our choice was to visit every design point with the functions $\psi_k$'s (in order not to loose information), some vectors in $(\w \beta_k)_{k=1,...,n}$  have a high variance and a large bias. To cancel their bad effect, we conduct step \ref{refresh}) by selecting the root $n$ vectors among the $\beta_k$'s that have the larger dependence criterion
\begin{align*}
\sum_{h,h'}\frac{ \left(p_{hh'}-\overline{p_{hh'}}^h\ \overline{ p_{hh'}}^{h'}\right)^2   }{\overline{p_{hh'}}^h\ \overline{ p_{hh'}}^{h'}} 
\end{align*} 
where $p_{h,h'}=\frac{1}{n} \sum_{i=1}^n \mathds{1}_{\{Y_i \in I_h\}} \mathds{1}_{\{(\beta_k^TX_i)\in J_{h'}\}}$ and $\overline{\ \cdot\ }^h$ is the mean over $h$. The partitions $(I_h)$ and $(J_h)$ have been defined having $\ulcorner \sqrt(n) \urcorner$ elements with equal sized (except the last). After this refinement we conduct step \ref{acp}), i.e. a PCA on the remaining vector $(\beta_k)_{k\in S}$. That is our final estimate of the \textit{index} is given by the $p$ eigenvectors of
\begin{align*}
\sum_{k\in S} \beta_k\beta_k^T
\end{align*}
associated with the $p$-largest eigenvalues.


\subsection{Simulations}

The ADETF method follows a typical semiparametric approach characterized by mild assumptions on the design but that requires the nonparametric estimation of the density. In a different spirit, a well known competitor is the approach called \textit{inverse regression} \cite{li1991}, that needs the \textit{linearity condition} (slightly weaker than ellipticity of the distribution of $X_1$). In the following simulation study, we compare the estimation of the index space $E$ given by ADE and ADETF with the one given by inverse regression methods, namely \textit{Sliced inverse regression} (SIR) \cite{li1991} and \textit{Sliced average variance estimation} (SAVE) \cite{cook1991}. One remarks that in the whole simulation study, the predictors are drawn from the Gaussian distribution. This is quite a comfortable situation for SIR and SAVE since they are not penalized by the restrictive framework they impose. For each estimate $\w E$ of $E$, we compute the estimation error with
\begin{align}\label{estimationerror}
\|\w P-P\|_\text{F},
\end{align}
where $ P$ (resp. $\w P$) is the orthogonal projector on $E$ (resp. $\w E$) and $\|\cdot\|_\text{F}$ is the Frobenius norm. In each situation, we assume that the dimension of $E$ is known.
\subsubsection{The models}

\paragraph{Model \cRM{1}.} We first consider
\begin{align*}
 Y= (\beta^TX) \sin\left( \beta^TX\right)+e,
\end{align*}
where $X= (X^{(1)},\ldots,X^{(p)})\overset{\dd}{=}\mathcal N (0,I)$, $e\overset{\dd}{=}\mathcal N (0,1)$. It is well known \cite{cook1991} that the SIR method fails when the link function is symmetric whereas SAVE achieves consistency. As a result, we run ADE, ADETF and SAVE on Model \cRM{1} with different values of the parameters $n$ and $p$. The boxplot are provided in Figure \ref{fig1}.

\paragraph{Model \cRM{2}.}From now we fix $p=6$ and $n=200$ (this illustrates situations quite difficult) and we focus on different link functions, each representing interested situations. In order to better understand how do the symmetries in the link function influence the methods, we generate 
\begin{align*}
 Y= \cos\left( \frac \pi 2 ( X^{(1)}-\mu)\right)+0.5e,
\end{align*}
with $\mu\in \R$. In our simulation, we try different values of $\mu$ from $0$, which correspond to a symmetric link function, to $1$. The boxplots are provided in Figure \ref{fig2}. 
\paragraph{Model \cRM{3}.}To highlight how the methods behave facing link functions with different level of fluctuations, we consider
\begin{align*}
\t{Model \cRM{3}:}\qquad \qquad Y= \tau \sin\left(  X^{(1)}/\tau\right)+0.5e,
\end{align*}
with $\tau\in \R$. For different values of $\tau$, we provide the boxplots of the errors in Figure \ref{fig3}. 

\paragraph{Model \cRM{4}.}We conclude by a two dimensional model defined as
\begin{align*}
\t{Model \cRM{4}:}\qquad \qquad Y=\frac{\sin(2 X^{(1)})}{.5+|1+X^{(2)}|} +\sigma e,
\end{align*}
where we found interesting to consider different values of $\sigma$. The method ADE does not appear because it only estimates a single direction.

\begin{figure}\centering
\includegraphics[width=14cm]{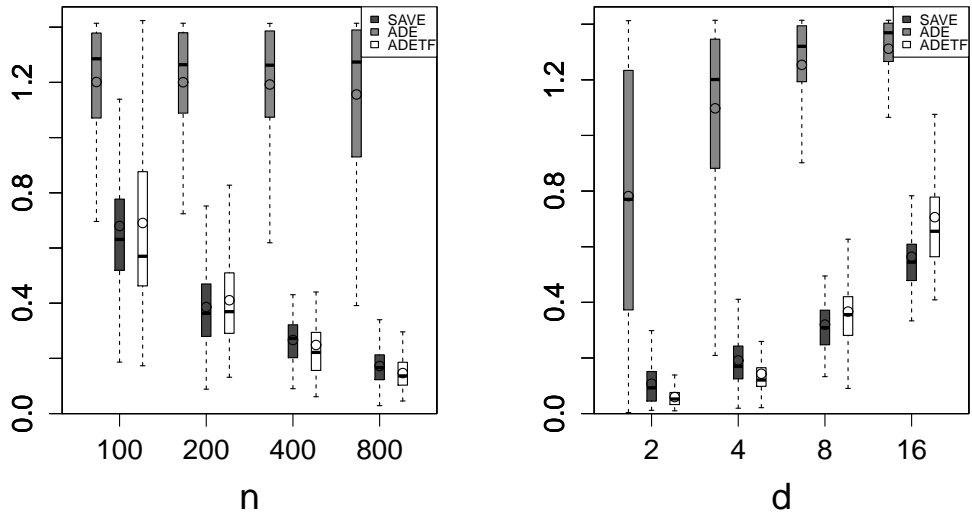} 
\caption{Boxplot over $100$ samples of the estimation error (\ref{estimationerror}) of SAVE, ADE and ADETF in the case of Model \cRM{1}, for different values of $d$ (when $n=400$) and different values of $n$ (when $d=6$).}\label{fig1}
\end{figure}

\begin{figure}\centering
\includegraphics[width=14cm]{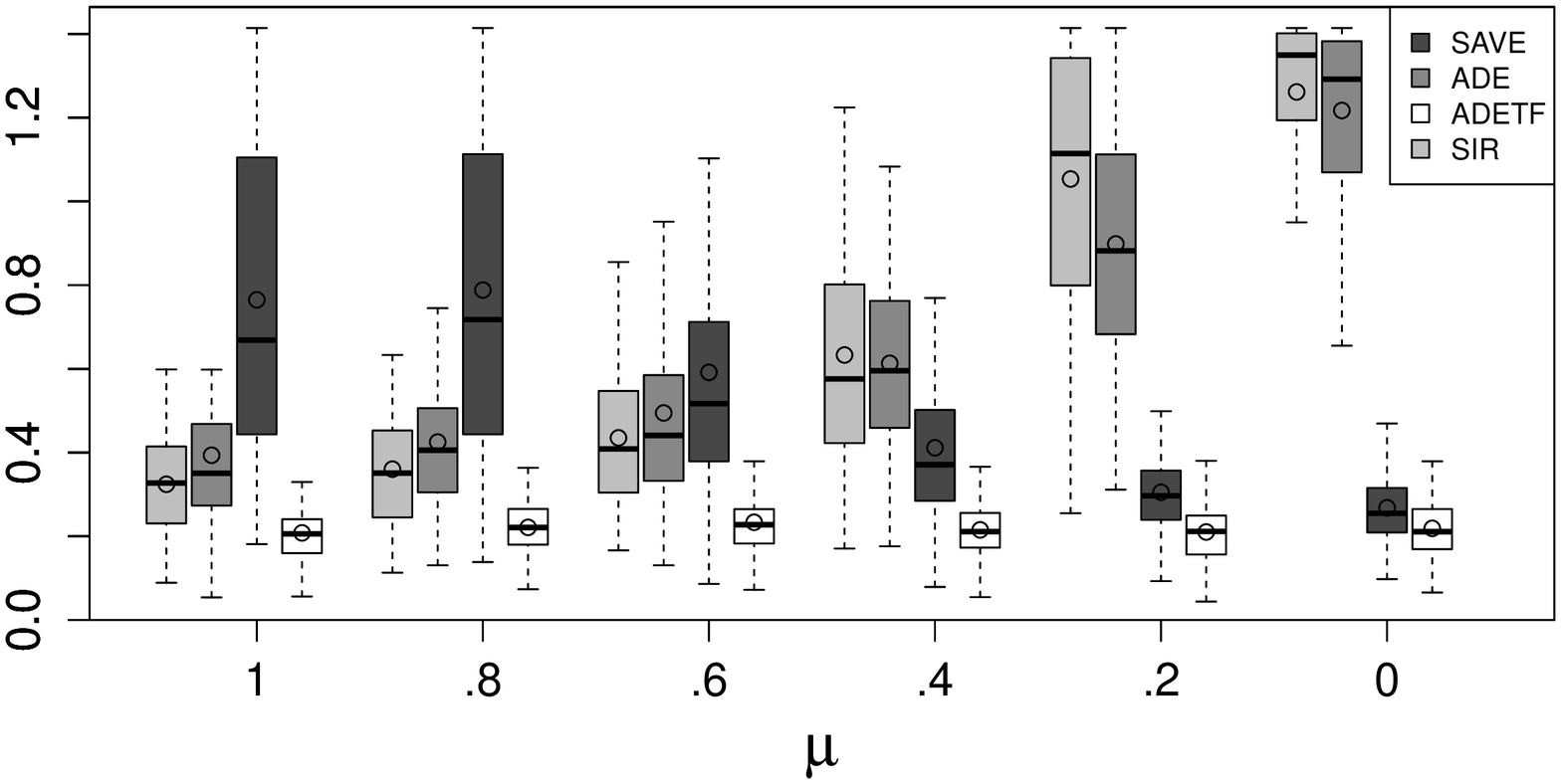} 
\caption{Boxplot over $100$ samples of the estimation error (\ref{estimationerror}) of SIR, SAVE, ADE and ADETF in the case of Model \cRM{2}, when $n=200$ and for different values of $\mu$.}\label{fig2}
\end{figure}

\begin{figure}\centering
\includegraphics[width=14cm]{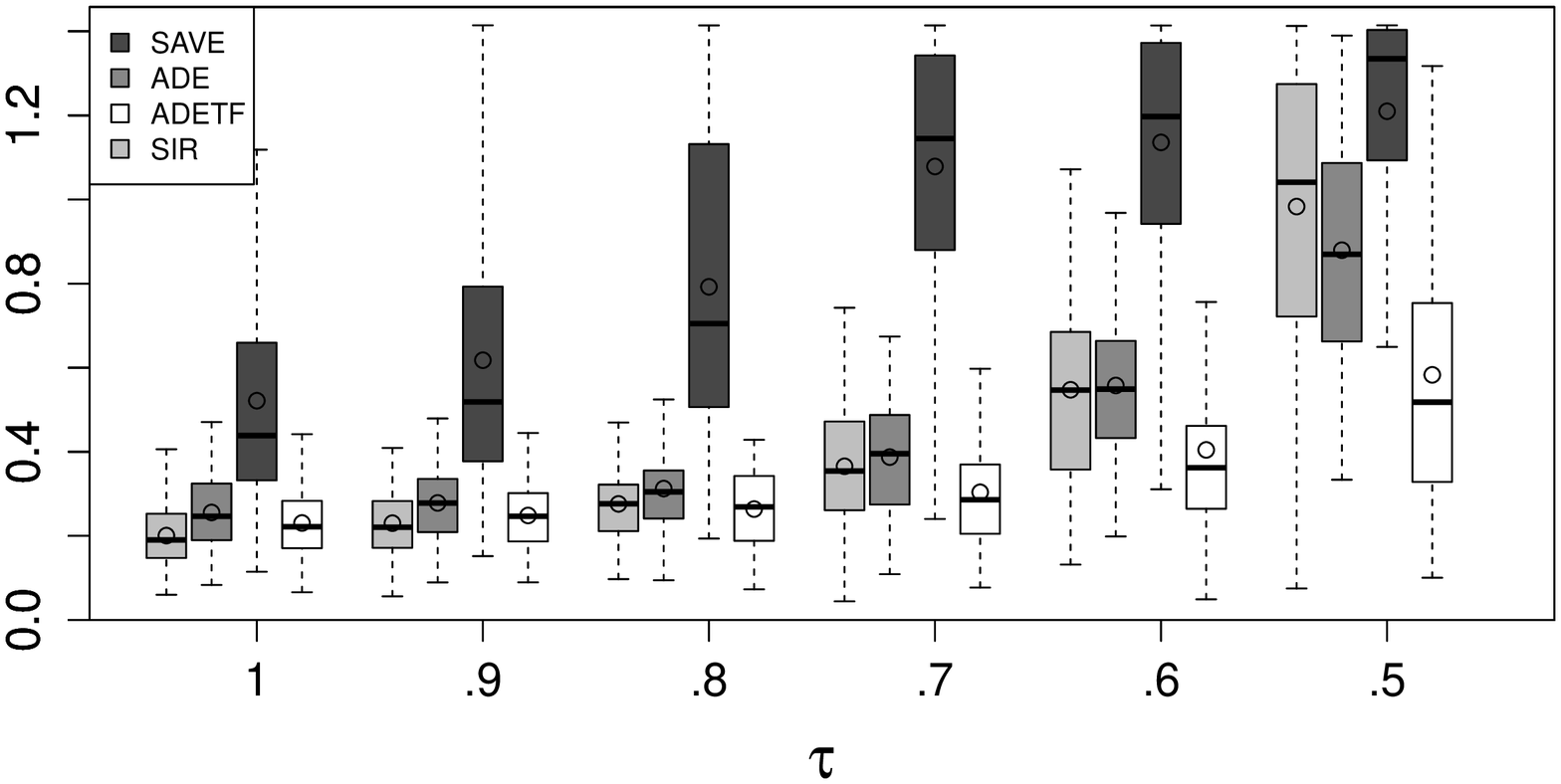} 
\caption{Boxplot over $100$ samples of the estimation error (\ref{estimationerror}) of SIR, SAVE, ADE and ADETF in the case of Model \cRM{3}, when $n=200$ and for different values of $\tau$.}\label{fig3}
\end{figure}

\begin{figure}\centering
\includegraphics[width=14cm]{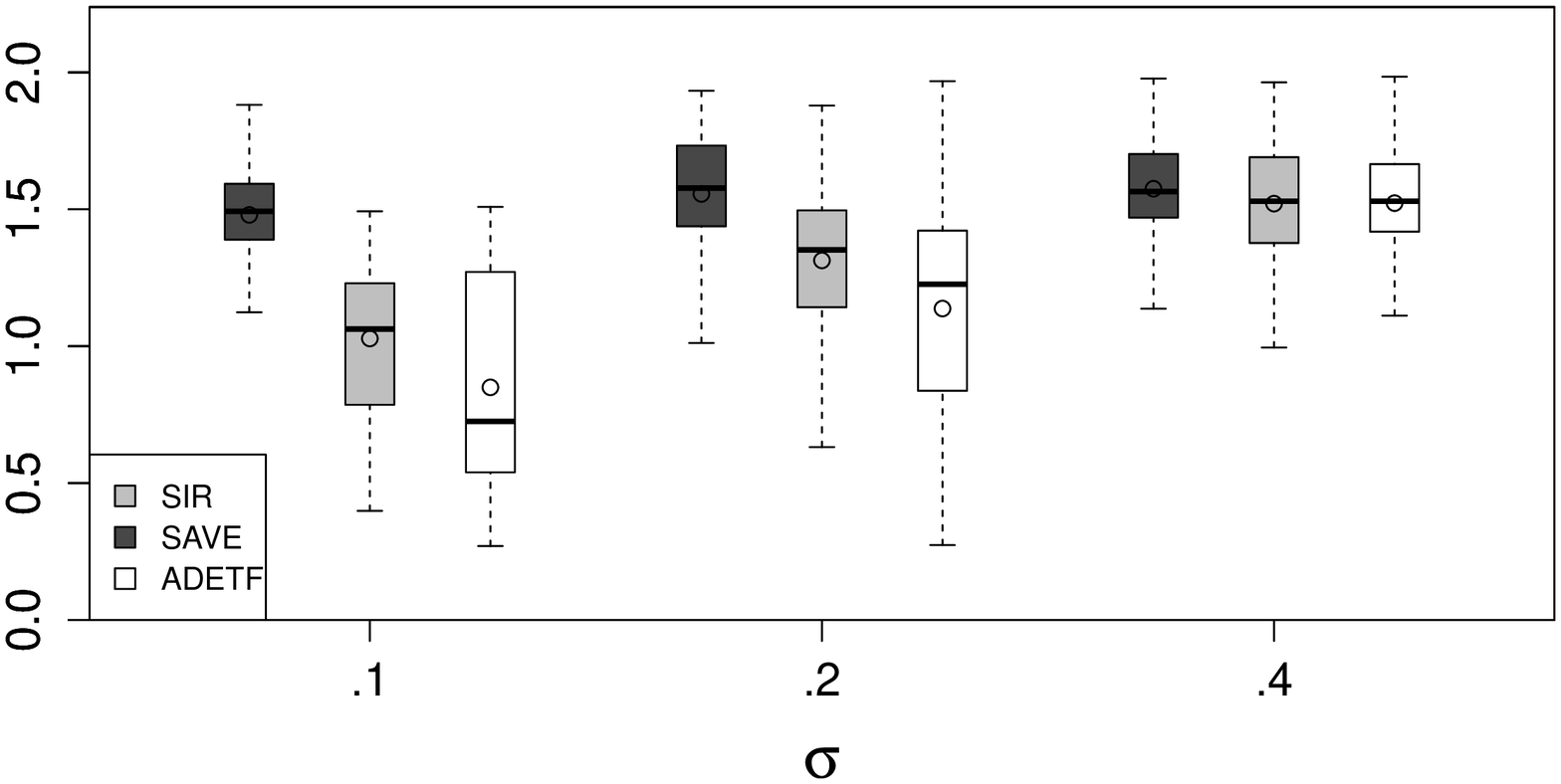} 
\caption{Boxplot over $100$ samples of the estimation error (\ref{estimationerror}) of SIR, SAVE, ADE and ADETF in the case of Model \cRM{4}, when $n=200$ and for different values of $\sigma$.}\label{fig4}
\end{figure}

\subsubsection{Interpretation of the results}

In figure \ref{fig1}, one remarks the accuracy of SAVE and ADETF whereas ADE fails completely to estimate the index. Asymptotically, ADETF becomes better than SAVE whereas SAVE seems to be more robust than ADETF when $d$ increase. The reason for this behavior when $d$ increases is the so called \textit{curse of dimensionality} raised in Remark $\ref{rem1}$. 

In figure \ref{fig2}, we analyse more in details how does the symmetry impact the methods. We remark that SIR and ADE produce similar poor estimate when the link function is symmetric. On the other hand, while SAVE is consistent in the presence of symmetry it seems to fail when the function is odd. Indeed, whereas SAVE and SIR and ADE seem to perform symmetrically with respect to the value of $\mu$, ADETF remains stable.

In figure \ref{fig3} and \ref{fig4}, we see that ADETF is more robust to the variation of the scale than other methods as SIR or ADE. In the two dimensional model, one may see that ADETF produce the better estimate for every level of noise considered.

\subsection{Adaptive ADE}

Unfortunately ADE and ADETF are subject to the so called \textit{curse of dimensionality}. As highlighted in Remark \ref{rem1}, the larger the dimension $d$ the smoother the density $f$ needs to be. Moreover, even if the density is smooth enough, one needs to use a high order kernel that may has poor performance at small sample size. In order to minimize bad effects of high dimension, we introduce the following adaptive strategy.

In \cite{juditsky2001} the authors proposed to estimate $\beta$
by an averaging of $\nabla g$ using a \textit{local linear estimator} \cite{fan1996} of $g$. 
In order to attain the root $n$ consistency, their estimator 
needs to be  improved via an adaptive procedure. 
 The idea is simple:
once $\beta$ is estimated, one could think of running once more the estimation procedure
in the reduced space in order to get advantage of the dimension reduction.
The point is this cannot be done exactly since the reduction space remains unknown; 
however the authors proved that using an estimate of $\beta$ with a suitable implementation, this idea is fruitful 
theoretically as well as practically.

All this is in theory not necessary in our case since, if $f$ is regular enough, 
the root $n$ consistency is achieved whatever the dimension, but we observe that
this refinement  procedure gives good results in practice.
Following their idea we notice that  for any test function $\psi$
\begin{align}\label{eqadaptivity}
\E\left[\frac {Y_1A\nabla \psi(AX_1)}{f_{|AX_1}(AX_1)}\right] = - \E\left[\frac {\nabla g (X_1) \psi(AX_1)}{f_{|AX_1}(AX_1)}\right]\in E \qquad\text{provided that }E\subset \spann (A),
\end{align}
where $f_{|AX_1}$ is the density of $AX_1$. 
For any $A$ we have the estimator
\begin{align}\label{yad}
\w \beta_\psi(A)=n^{-1} \sum_{i=1}^n \frac{Y_iA \nabla \psi ( A X_i ) )}{\w f_{|AX_1}(A X_i)}  ,
\end{align}
with 
\begin{align*}
\w f_{|AX_1}(x)= (nh^d)^{-1} \sum_{i=1}^n K(h^{-1} (AX_i-x)),\qquad \text{for every } x\in \R^p.
\end{align*} 
After an initial estimation $\w \beta$ obtained with $A=Id$ and several test functions 
$\psi_1,\dots \psi_K$,
we take $A= \w \beta\w \beta^T +\epsilon I$ as in \cite{juditsky2001} and obtain 
a second estimator whose window has been stretched in the interesting direction, 
i.e. the direction where $g$ varies. This procedure might be iterated several times with $h$ and $\epsilon$ decreasing. 

The theoretical study and the implementation details require much more work that seems to be beyond the scope of the present paper. This could be done following the well documented semiparametric literature on the subject \cite{juditsky2001}, \cite{dalalyan2008} and \cite{xia2007}.

\section{A remark about the generalization of Theorem 1}\label{generalizing}

In view of the intriguing convergence rates stated in Theorem 
\ref{thelemma}, one may be curious to know the behavior of our 
estimator when estimating more general functionals with the form 
\begin{align*}
 I_T= \int_{}T(x,f(x))dx,
\end{align*}
where $T:\R^{d}\times \R^+ \r \R$ is such that $y\mapsto T(x,y)$ has a second order derivative bounded uniformly on $x$. 
Following the approach of Section \ref{s1}, the estimator we consider is 
\begin{align}\label{formestimate}
\w 
I_{T} = n^{-1} \sum_{i=1}^n \frac{T\big(X_i,\w f^{(i)}(X_i)\big)}{\w f^{(i)}(X_i)}.
\end{align}
The study of the asymptotic behavior of $\sqrt n (\w I_{T}- I_{T})$ generalizes Theorem \ref{thelemma}. It turns out that the case $ T:(x,y)\mapsto \varphi(x)$ is the only case where the rates are faster than root $n$. For other functionals, $\sqrt n (\w I_{T}- I_{T})$ converges to a normal distribution. In view of the negative aspect of the following results with respect to those of Theorem 
\ref{thelemma}, we provide an 
informal calculation that leads to the asymptotic law of $\sqrt n (\w 
I_{T}-I_{T})$. By assumption on $T$, using a Taylor expansion with respect to the second coordinate of $T$, we have
\begin{align*}
n^{1/2} (\w 
I_{T}-I_{T}) =  n^{-1/2} \sum_{i=1}^n \left( \frac{T(X_i,f(X_i))}{\w f_i}- I_{T}+\frac{\partial _y T(X_i,f(X_i)) 
( \w f_i-f(X_i))}{\w f_i}\right)  +\w R_2,
\end{align*}
with 
\begin{align*}
|\w R_2| \leq C n^{-1/2}\sum_{i=1}^n \frac{(\w f_i-f(X_i))^2}{\w f_i}= O_{\P} (n^{1/2}h^{2r}+ n^{-1/2}h^{-d})
\end{align*}
due to equations (\ref{unifb}) and (\ref{sansnom}). Then, we write
\begin{align*}
\sqrt n (\w I_{T}-I_{T}) = \w R_0+\w R_1+\w R_2,
\end{align*}
with
\begin{align*}
\w R_0&=  n^{-1/2} \sum_{i=1}^n \frac{T(X_i,f(X_i))}{\w f_i} -I_T -\frac{\partial _y T(X_i,f(X_i)) 
f(X_i)}{\w f_i}+\int \partial _y T(x,f(x)) 
f(x)dx \\
\w R_1&= n^{-1/2} \sum_{i=1}^n \partial _y T(X_i,f(X_i))-\int \partial _y T(x,f(x)) 
f(x)dx.
\end{align*}
Provided Theorem \ref{thelemma} can be applied two times, we show that $\w R_0=o_{\P}(1)$. As a consequence $\sqrt n (\w I_{T}-I_{T})=o_{\P}(1)$ if and only if the variance of $\w R_1$ is degenerate, that is equivalent to
\begin{align*}
 \partial _y T(X_i,f(X_i)) = c \qquad \text{a.s.}
\end{align*}
If we want this to be true for a reasonably large class of distribution function, it would imply
\begin{align*}
 \partial _y T(x,y) = c \qquad \text{for all } (x,y)\in \R^d\times \R^+, 
\end{align*}
the solutions have the form $T(x,y)= \varphi(x) + cy$.

\section{Concluding remarks}
%


There exists some links between Theorem \ref{thelemma2} and nonparametric estimation. Those links are beyond the scope of this article but can be the subject of further research. Indeed, Theorem \ref{thelemma2} is not so far from dealing with nonparametric regression. On the one hand, one can use it for the estimation of the Fourier (or wavelet) coefficient in the $L_2$ expansion of $g$
\begin{align*}
\w c_k(g) = n^{-1}\sum_{i=1}^n \frac{Y_i \psi_k(X_i)}{\w f^{(i)}(X_i)},
\end{align*}
($\psi_k$ is the Fourier $L_2$ basis), this would lead to projection estimates
\begin{align}\label{conv1}
\sum_{k=1}^{K}\w c_k(g) \psi_k(y),
\end{align}
and estimates by shrinkage. On the other hand, one can similarly define the kernel estimator
\begin{align}\label{conv2}
\sum_{i=1}^n \frac{Y_i K_{h_2}(X_i-x)}{\sum_{j=1}^n K_{h_1}(X_j-X_i) },
\end{align}
where $h_1$ and $h_2$ are bandwidths each linked with the estimation of $f$ and the regularization of $g$, respectively. Similar estimators of the regression function have already been introduced in the case of unknown random design (density $f$). Estimate (\ref{conv1}) is linked with the estimate (3.3.6) p.51 of \cite{hardle1990}, studied in \cite{rutkowski1982}, whereas estimate (\ref{conv2}) is reminiscent of the Gasser-Muller estimator \cite{gasser1979}. Both latter estimates are called convolution estimator of the regression because they estimate directly $<g,K_h(\cdot-y)>$ whereas the most popular approach, inspired by the Naradaya-Watson estimate, has been to estimate separately $<g f,K_h(\cdot-y)>$ and $<f,K_h(\cdot-y)>$ by simple empirical means, $\w{gf}$ and $\w{f}$ respectively, and then to estimate $g$ by $\w{ g f}/\w f$. It would be interesting to understand how Equation (\ref{conv1}) or (\ref{conv2}) could improve the estimation of $g$, work along this line is under progress.

\section{Proofs}

\subsection{Proof of Theorem \ref{thelemma}}

For clarity, we introduce the following notation
\begin{align*}
&K_{ij}=h^{-p}K(h^{-1}(X_i-X_j))\\
&\w f_i =\frac1{n-1} \sum_{j\neq i}^n K_{ij}\\
&\w v_i = \frac1{(n-1)(n-2)} \sum_{j\neq i}^n (K_{ij}- \w f_i)^2 ,
\end{align*}
and for any function $g :\R^p \r \R$, we define
\begin{align} \label{noth}
 g_{h}(x) = \int g(x+hu)K(u)du.
\end{align}
We start by showing (\ref{borne2}), then (\ref{borne1}) will follow straightforwardly.

\noindent{\textbf{Proof of (\ref{borne2}):}} 

The following development reminiscent of the Taylor expansion
\begin{align*}
\frac {1}{\w f_i } =  \frac{1} {f_h(X_i)} 
+ \frac {f_h(X_i) -\w f_i } {f_h(X_i)^2}
+\frac {(f_h(X_i) -\w f_i )^2} {f_h(X_i)^3}
 +\frac {(f_h(X_i)-\w f_i )^3} {\w f_i f_h(X_i)^3},
\end{align*}
allows to expand our estimator as a sum of many terms where the density estimate $\w f_i$ is moved to the numerator,
with the exception of the fifth one. We will show that this last term goes quickly to $0$. For the linearised terms, this is very messy because the correct bound will be obtained by expanding also $\w f_i$ in those expressions. 
In order to sort out these terms, we borrow to Vial \cite{vial2003} the trick 
of making appear a degenerate $U$-statistic in such a development (by inserting the right quantity in $\w R_0$ below).
More explicitly, recalling that
\begin{align*}
n^{1/2} \left(\w I_{cor}(\varphi)
 -\int \varphi(x)  dx \right)= n^{-1/2} \left(\sum_{i=1}^n \frac{\varphi(X_i)}{\w f_i}\Big(1-\frac{\w v_i}{\w f_i^2}\Big)
-\int \varphi(x) dx\right),
\end{align*}
using the notations
\begin{align*}
&\psi_q(x)=\frac{\varphi(x)}{f_h(x)^q},~~~q\in\mathbb N,\\
&\tilde\psi_1(x)=\left(\varphi(x)\frac{f(x)}{f_h(x)^2}\right)_h ,
\end{align*}
we obtain
\begin{align}\label{decomplem1}
n^{1/2} \left(\w I_{cor}(\varphi)
 -\int \varphi(x)  dx \right)=\w R_0+\w R_1+\w R_2+\w R_3+\w R_4+\w R_5
\end{align}
with (we underbrace terms which have been deliberately introduced and removed)
\begin{align*}
&\w R_0 = n^{-1/2} \sum_{i=1}^n \psi_1(X_i)-\psi_{2}(X_i) {\w f_i }  
+\underbrace{\tilde\psi_1(X_i)} -\underbrace{\E[\psi_1(X_i)]}\\
&\w R_1 = \int \Big(\underbrace{f(x)f_h(x)^{-1}}-1\Big)\varphi(x) \,dx\\
&\w R_2= n^{-1/2} \sum_{i=1}^n  \psi_1(X_i)-\underbrace{\tilde\psi_1(X_i)} \\
&\w R_3 = n^{-1/2} \sum_{i=1}^n  \psi_3(X_i)\{ ( f_h(X_i)-\w f_i )^2-\underbrace{\w v_i}\} \\
&\w R_4 =  n^{-1/2}\sum_{i=1}^n \frac{\psi_3(X_i) 
\w v_i}{\w f_i ^3 }(\,\underbrace{\w f_i ^3}- f_h(X_i)^3)\\
&\w R_5 = n^{-1/2} \sum_{i=1}^n\psi_3(X_i) \frac {(f_h(X_i)-\w f_i )^3} {\w f_i }.
\end{align*}
$\w v_i$ appears to be a centering term in $\w R_3$.
We shall now compute bounds for each term separately. 
Since some of these bound will be used for the proof of (\ref{borne1})
we shall use only the property
\begin{align*}
h^{s} + n^{1/2}h^{r} +  n^{-1/2}h^{-d}\rightarrow 0.
\end{align*}

\paragraph{Step 1\,:}$\|\w R_0\|_2=O(n^{-1/2} h^{-d/2})$.
Remark that 
\begin{align*}
\w R_1  &= n^{-1/2} (n-1)^{-1}\sum_{i\ne j}\E[u_{ij}|X_j]-u_{ij}+E[u_{ij}|X_i]-E[u_{ij}],
\end{align*}
with $u_{ij}=\psi_{2}(X_i)K_{ij}$, is a degenerate $U-$statistic. The $n(n-1)$ terms in the sum are all orthogonal with $L_2$ norm smaller than $\|u_{ij}\|_2$, hence
\begin{align*}
(n-1)E[\w R_1^2]\le &  \E[u_{12}^2]
\le \|\psi_2\|_\infty^2E[K_{12}^2]
\end{align*}
and
\begin{align}
\E[K_{12}^2|X_1]\le &h^{-2d}\int K(h^{-1}(x-X_1))^2f(x)dx 
\le h^{-d}\|f\|_\infty \int K(u)^2du.\label{K2}
\end{align}

\paragraph{Step 2\,:}$\w R_1= O( n^{1/2}h^r)$. This classically results from Equation~(\ref{bochner2}) of Lemma~\ref{bochner},
and from Assumption (B\ref{ash3}).

\paragraph{Step 3\,:}$\|\w R_2\|_2= O( n^{1/2}h^r+h^{{s}})$. We can rearrange the function $\psi_1(x)-\tilde\psi_1$
as 
\begin{align*}
\psi_1(x)-\tilde\psi_1(x)= &\Big(\psi_1(x)-\psi_{1h}(x)\Big)+\Big(\psi_{1h}(x)-\tilde\psi_1(x)\Big)
\end{align*}
(with the notation (\ref{noth})) and since
\begin{align*}
\|\psi_{1h}(x)-\tilde\psi_1(x)\|_\infty
=& \|\Big(\psi_1(x)-\varphi(x)\frac{f(x)}{f_h(x)^2}\Big)_n\|_\infty\\
\le& \|\psi_1(x)-\varphi(x)\frac{f(x)}{f_h(x)^2}\|_\infty\\
=& \Big\|\frac{\varphi}{f_h^2}(f_h-f)\Big\|_\infty
\end{align*}
we have
\begin{align*}
\w R_2\le n^{-1/2}\Big|\sum_{i=1}^n\psi_{1h}(X_i)-\psi_{1}(X_i)\Big|
+n^{1/2}\Big\|\frac{\varphi}{f_h^2}\Big\|_\infty\,\| f_h-f\|_\infty
\end{align*}
and we conclude with Equations~(\ref{bochner1}) and (\ref{bochner2}) of Lemma~\ref{bochner}.

\paragraph{Step 4\,:}$\|\w R_3\|_2=O(n^{-1/2}h^{-d/2})$.
We first rewrite separately each term. Set
\begin{align*}
&U_i=   (f_h(X_i)-\w f_i )^2-\w v_i, 
\end{align*}
and rewrite $\w R_3$ as
\begin{align*}
&\w R_3 = n^{-1/2} \sum_{i=1}^n\psi_3(X_i) U_i.
\end{align*}
Consider a sequence of real numbers $(x_j)_{1\le j\le p}$ and set
\begin{align*}
&m=\frac{1}{p}\sum_{j=1}^px_j\\
&v=\frac{1}{p(p-1)}\sum_{j=1}^p(x_j-m)^2=\frac{1}{p(p-1)}\sum_{j=1}^p(x_j^2-m^2),
\end{align*}
then 
\begin{align*}
m^2-v=\Big(1+\frac{1}{p-1}\Big)m^2-\frac{1}{p(p-1)}\sum_{j=1}^px_j^2=\frac{2}{p(p-1)}\sum_{j< k}x_jx_k.
\end{align*}
Applying this with $x_j=K_{ij}-f_h(X_i)$ ($i$ is fixed) and $p=n-1$ we get
\begin{align*}
U_i=&\frac2{(n-1)(n-2)}\sum_{j\ne i,k\ne i,j< k}(K_{ij}-f_h(X_i))(K_{ik}-f_h(X_i))\\
=&\frac2{(n-1)(n-2)}\sum_{j< k}\xi_{ij}\xi_{ik},
\end{align*}
with 
\begin{align*}
&\xi_{ij}=K_{ij}-f_h(X_i)\\
&\xi_{ii}=0.
\end{align*}
Then
\begin{align*}
&\w R_3 =  \frac{2n^{-1/2}}{(n-1)(n-2)}\sum_i\sum_{j< k}\psi_3(X_i)\xi_{ij}\xi_{ik}.
\end{align*}
We are going to calculate $\E[\w R_3^2]$ by using the Efron-Stein inequality (Theorem~\ref{efronstein}) 
and the moment inequalities  (\ref{eq1}) to (\ref{eq3}) for 
$\xi_{ij}$ stated in Lemma \ref{moments}; in particular, by (\ref{eq1}) $\E[\w R_3^2]=Var(\w R_3)$. 
Consider $\w R_3=f(X_1,\dots X_n)$ as a function of the $X_i$'s and define
\begin{align*}
&\w R'_3 = f(X'_1,X_2\dots X_n)\\
 &\xi'_{1j}=h^{-d}K(h^{-1}(X'_1-X_i))-f_h(X_1)\\
&\xi'_{i1}=h^{-d}K(h^{-1}(X'_1-X_i))-f_h(X_i)\\
&\xi'_{ij}=\xi_{ij}\t{ if }i\ne 1 \t{ and }j\ne 1
\end{align*}
where $X'_1$ is a copy of $X_1$ independent from the sample $(X_1,\ldots,X_n)$. Then by the Efron-Stein inequality and the triangular inequality
\begin{align*}
\|\w R_3\|_2\le& \left( \frac {n }{ 2} \right)^{1/2} \|\w R_3-\w R'_3\|_2 \\
= &Cn^{-2}\left\| \sum_{j< k}(\psi_3(X_1)\xi_{1j}\xi_{1k}-\psi_3(X'_1)\xi_{1j}'\xi_{1k}')
+\sum_i\sum_{1< k}\psi_3(X_i)(\xi_{i1}-\xi'_{i1})\xi_{ik}\right\|\\
\le &Cn^{-2}\left(\left\| \sum_{j< k}\psi_3(X_1)\xi_{1j}\xi_{1k}-\psi_3(X'_1)\xi_{1j}'\xi_{1k}'\right\|
+\left\|\sum_{1< k}\sum_i\psi_3(X_i)(\xi_{i1}-\xi'_{i1})\xi_{ik}\right\|\right)\\
= & Cn^{-2}(\|T_1\|_2+\|T_2\|_2).
\end{align*}
Remember that $\xi_{ii}=0$. Noting that the terms in the first sum are orthogonal (by independence of $\xi_{ij}$ and $\xi_{ik}$ conditionally to $X_i$ and (\ref{eq1})) we obtain
\begin{align*}
\|T_1\|_2= &  \sqrt{\tfrac{(n-1)(n-2)}{2}} \|\psi_3(X_1)\xi_{12}\xi_{13}-\psi_3(X'_1)\xi_{12}'\xi_{13}'\|_2\\
\le& \sqrt2\,n \|\psi_3\|_{\infty}\|\xi_{12}\xi_{13}\|_2\\
=&\sqrt 2\,n \|\psi_3\|_{\infty}\E[\E[\xi_{12}^2\xi_{13}^2|X_1]]^{1/2}\\
= & \sqrt2\,n\|\psi_3\|_{\infty}\|\E[\xi_{12}^2|X_1] \|_2\\
\leq& Cnh^{-d}
\end{align*}
by (\ref{eq2}). Because the terms of the second sum
are orthogonal whenever the values of $k$ are different, we get
\begin{align*}
\|T_2\|_2=& (n-1)^{1/2}\Big\|\sum_i\psi_3(X_i)(\xi_{i1}-\xi'_{i1})\xi_{i2}\Big\|.
\end{align*}
By first developing and then using that $X_1'$ is an independent copy of $X_1$, we obtain 
\begin{align*}
 \Big\|\sum_i\psi_3(X_i)(\xi_{i1}-\xi'_{i1})\xi_{i2}\Big\|_2^2 
 \le& n \E\left[\psi_3(X_3)^2(\xi_{31}-\xi'_{31})^2\xi_{32}^2\right]|\\
 &+n^2|\E\left[\psi_3(X_3)\psi_3(X_4)(\xi_{31}-\xi'_{31})\xi_{32}(\xi_{41}-\xi'_{41})\xi_{42}\right]|\\
 \leq& nC \E\left[(\xi_{31}-\xi'_{31})^2\xi_{32}^2\right]\\
 &+n^2 C'\E\left[|\E[(\xi_{31}-\xi'_{31})\xi_{32}(\xi_{41}-\xi'_{41})\xi_{42}|X_3,X_4]|\right]\\
= & 2C n\E\left[\xi_{31}^2\xi_{32}^2\right]+2Cn^2 \E\left[|\E[
\xi_{31}\xi_{32}\xi_{41}\xi_{42}|X_3,X_4]|\right].
\end{align*}
Then by (\ref{eq2}) $\E\left[\xi_{31}^2\xi_{32}^2\right]=\E\left[\E[\xi_{31}^2|X_3]^2\right]\le Ch^{-2d}$ and by (\ref{eq3})
\begin{align*}
\E[|\E[\xi_{31}\xi_{32}\xi_{41}\xi_{42}|X_3,X_4]|]=&\E[\E[\xi_{31}\xi_{41}|X_3,X_4]^2]\\ \le& 2\|f\|_\infty^2 h^{-2d}\E[K_2(h^{-1}(X_4-X_3))^2]+2\|f\|_{\infty}^4\\
\le &2\|f\|_\infty^3 h^{-d}\int K_2(u)^2d u+2\|f\|_{\infty}^4.
\end{align*}
Bringing everything together
\begin{align*}
&\|\w R_3\|_2\leq C n^{-1}h^{-d}+C n^{-1}h^{-d}+C n^{-1/2}h^{-d/2}= O(n^{-1/2}h^{-d/2})
\end{align*}
because $nh^d\r\infty$.

\paragraph{Step 5\,:}$\w R_4=O_{\P}(n^{-1} h^{-3d/2})$. 
We start with a lower bound for $\w f_i$ by proving the existence of $N(\omega)$ such that
\begin{align}
\forall\,n\ge N(\omega),~\forall\, i,~~\frac{b}{2}<\w f_i<2\|f\|_\infty.\label{unifb}
\end{align}
Notice that
\begin{align*}
&\w f_i=\frac{n}{n-1}\left(\w f(X_i)-\frac{h^{-d}}{n-1}K(0)\right)\\
&\w f(x)=\frac1{nh^d}\sum_{k=1}^nK(h^{-d}(x-X_k)),
\end{align*}
due to the almost sure uniform convergence of $\w f$ to $f$ (Theorem~1 in \cite{devroye1980})
we have for $n$ large enough 
\begin{align*}
\frac{2b}{3}<\inf_{x\in Q} \w f (x) \le
\sup_{x\in Q} \w f (x)  < \frac{3}{2}\|f\|_\infty 
\end{align*}
and since assumption $nh^d\rightarrow\infty$, (\ref{unifb}) follows.
We can now compute the expectation of $\w R_4$ restricted to $\{n\ge N(\omega)\}$. Because
\begin{align*}
|\w R_4|1_{n>N(\omega)} \le&C   n^{-1/2}\sum_{i=1}^n |\w f_i- f_h(X_i)|\w v_i
\end{align*}
we have by the Cauchy-Schwartz inequality
\begin{align}
\E[|\w R_4|1_{n>N(\omega)}] \le &C   n^{1/2}\E[(\w f_1- f_h(X_1))^2]^{1/2}\E[\w v_1^2]^{1/2}.\label{st50}
\end{align}
Applying the fact that for any real number $a$, $ \frac 1 p  \sum_{j=1}^p(x_j-\overline x)^2\leq \frac 1 p  \sum_{i=1}^p(x_j-a)^2$ to $x_j=K_{1j}$, $p=n-1$ and $a=f_h(X_1)$, we obtain that
\begin{align*}
&\w v_1 \le \frac1{(n-1)(n-2)} \sum_{j=2}^n \xi_{1j}^2, 
\end{align*}
then using (\ref{eq2})
\begin{align}\nonumber
\E[\w v_1^2] 
=& (n-1)^{-1}(n-2)^{-2}\E[\xi_{12}^4]+(n-1)^{-1}(n-2)^{-1}\E[\xi_{12}^2\xi_{13}^2] \\\nonumber
\le& C'n^{-3}h^{-3d}+C'n^{-2}h^{-2d}\\
\le& C''n^{-2}h^{-2d}\label{st51}
\end{align}
because $nh^d$ is lower bounded. On the other hand using (\ref{eq2})
\begin{align}\label{st52}
\E[(\w f_1- f_h(X_1))^2]=&\frac1{n-1}\E[\xi_{1i}^2]\le Cn^{-1}h^{-d}.
\end{align}
Putting together (\ref{st50}), (\ref{st51}) and (\ref{st52}),
\begin{align*}
\E[|\w R_4|1_{n>N(\omega})] \le &C   n^{1/2}n^{-1}h^{-d}n^{-1/2}h^{-d/2}=C n^{-1}h^{-3d/2}.
\end{align*}
In particular
\begin{align*}
\P(n h^{3d/2}|\w R_4|>A)
\le&\P(n h^{3d/2}|\w R_4|1_{n>N(\omega)}>A)+\P(n\le N(\omega))\\
\le&  C\,A^{-1}+\P(n\le N(\omega)).
\end{align*}
This proves the boundedness in probability of $n h^{3d/2}|\w R_4|$.

\paragraph{Step 6\,:}$\w R_5= O_\P( n^{-1}h^{-3d/2}+n^{-3/2} h^{-2d})$.
Following (\ref{unifb}) since
\begin{align*}
|\w R_5|1_{n>N(\omega)} \leq 2b^{-3}\|\varphi\|_\infty n^{-1/2} \sum_{i=1}^n{|\w f_i -f_h(X_i)|^3},
\end{align*}
we can show the convergence in probability of the right-hand side term as in Step 5. We have
indeed by the Rosenthal's inequality\footnote{\label{foot}For a martingale $(S_i,\mathcal{F}_i)_{i\in \N}$ and $2\leq p< +\infty$, we have $\E[|S_n|^p]\leq C\{\E[(\sum_{i=1}^n\E[X_i^2|\mathcal{F}_{i-1}])^{p/2}]+\sum_{i=1}^n \E|X_i|^p\}$, where $X_i = S_i -S_{i-1}$ (see for instance \cite{hall1980}, p. 23-24).}
\begin{align}\nonumber
\E\left[n^{-1/2} \sum_{i=1}^n|\w f_i -f_h(X_i)|^p\right]=&n^{1/2}(n-1)^{-p}\E[|\sum_{i=2}^n \xi_{1i} |^p]\\
&\leq  Cn^{1/2}n^{-p}\{(n\E[\xi_{12}^2])^{p/2}+n\E[|\xi_{12}|^p]\}\nonumber\\
&\leq  C'\{n^{(1-p)/2}h^{-pd/2}+n^{3/2-p}h^{-(p-1)d}\}\label{rose}.
\end{align}
(cf. (\ref{eq2})). Hence with $p=3$
\begin{align*}
\E\left[|\w R_5|1_{n>N(\omega)}\right] \leq C\{n^{-1}h^{-3d/2}+n^{-3/2}h^{-2d}\}
\end{align*}
and we conclude as in Step 5.

\noindent{\textbf{Proof of (\ref{borne2}):}}  Putting together the steps 1 to 6, and taking into account, concerning $\w R_5$, 
that $n^{-3/2} h^{-2d}=(n^{-1/2} h^{-d/2})(n^{-1} h^{-3d/2})$, we obtain (\ref{borne2}). 

For (\ref{borne1}), we use a shorter expansion
which leads to an actually much simpler proof:
\begin{align*}
\frac {1}{\w f_i } =  \frac{1} {f_h(X_i)} 
+ \frac {f_h(X_i) -\w f_i } {f_h(X_i)^2}
+\frac {(f_h(X_i) -\w f_i )^2} {\w f_i f_h(X_i)^2}
\end{align*}
and
\begin{align*}
r= n^{-1/2} \Big(\sum_{i=1}^n \frac{\varphi(X_i)}{\w f_i} -\int \varphi(x) dx \Big) 
=\w R_0+\w R_1+\w R_2+\w R'_5
\end{align*}
with 
\begin{align*}
&\psi_q(x)=\frac{\varphi(x)}{f_h(x)^q},~~~q\in\mathbb N\\
&\w R_0 = \int \Big(\underbrace{f(x)f_h(x)^{-1}}-1\Big)\varphi(x) \,dx\\
&\w R_1 = n^{-1/2} \sum_{i=1}^n \psi_1(X_i)-\psi_{2}(X_i) {\w f_i }  
+\underbrace{\tilde\psi_1(X_i)} -\underbrace{\E[\psi_1(X_i)]},~~~\tilde\psi_1(x)=\left(\varphi(x)\frac{f(x)}{f_h(x)^2}\right)_h\\
&\w R_2= n^{-1/2} \sum_{i=1}^n  \psi_1(X_i)-\underbrace{\tilde\psi_1(X_i)} \\
&\w R'_5 = n^{-1/2} \sum_{i=1}^n\psi_2(X_i) \frac {(f_h(X_i)-\w f_i )^2} {\w f_i }.
\end{align*}
The term $\w R'_5$ is bounded exactly as $\w R_5$ but since now we use (\ref{rose}) with 
$p=2$ instead of $p=3$, we obtain
\begin{align*}
\E[|\w R'_5|1_{n>N(\omega)}) \leq C n^{1/2} \E[|f_h(X_1)-\w f_1 |^2]\le Cn^{-1/2} h^{-d}
\end{align*}
and we get $|\w R_5'|=O_\P(n^{-1/2} h^{-d})$.
\qed

\subsection{Proof of the Theorem \ref{thelemma2}}
By decomposition (\ref{decomp}), we are interested in the asymptotic law of the vector 
\begin{align*}
 n^{-1/2} \sum_{i=1}^n \frac{\sigma(X_i) \psi(X_i)}{\w f_i} e_i + \ n^{-1/2} \left(\sum_{i=1}^n \frac{g(X_i) \psi(X_i)}{\w f_i}-\int g(x) \psi(x)d x\right).
\end{align*}
By Lemma \ref{thelemma}, the right hand-side term goes to $0$ in probability. For the other term, we use the decomposition $\w S_1 + \w S_2$, with 
\begin{align}\label{decomplem2}
\w S_1 = n^{-1/2} \sum_{i=1}^n \frac{s(X_i)}{ f(X_i)} e_i\quad \text{ and}\quad  \w S_2 =n^{-1/2} \sum_{i=1}^n \frac{s(X_i)( f(X_i)-\w f(X_i))}{\w f_i f(X_i)} e_i.
  \end{align}   
where $s(X_i) =\sigma(X_i) \psi(X_i)$. We define $\mathcal F $ as the $\sigma$-field generated by the set of random variables $\{X_1,X_2,\cdots \}$. We get
\begin{align*}
\E[\w S_2^2 |\mathcal F]= n^{-1} \sum_{i=1}^n \frac{s(X_i)^2( f(X_i)-\w f_i)^2}{\w f_i^2 f(X_i)^2},
\end{align*} 
then, one has
\begin{align*}
\E[\w S_2^2 |\mathcal F]\leq (b^2\inf_{i} \w f_i^2)^{-1} \|s\|_{\infty}^2 n^{-1}\sum_{i=1}^n ( f(X_i)-\w f_i)^2.
\end{align*} 
For the term on the left, since $s$ has support $Q$ we can use (\ref{unifb}), that is for $n$ large enough, it is bounded. For the right hand-side term, it follows that
\begin{align*}
 n^{-1}\sum_{i=1}^n ( f(X_i)-\w f_i)^2\leq 2(n^{-1}\sum_{i=1}^n ( f(X_i)-f_h(X_i))^2+n^{-1}\sum_{i=1}^n (f_h(X_i)-\w f_i)^2),
\end{align*} 
and then using Lemma \ref{bochner} and (\ref{rose}) for $p=2$ we provide the bound
\begin{align}\label{sansnom}
 \|n^{-1}\sum_{i=1}^n ( f(X_i)-\w f_i)^2\|_1\leq C\{h^{2r}+n^{-1}h^{-d}\} .
\end{align}
Therefore, we have shown that 
$\E[\w S_2^2 |\mathcal F]\r 0 $ in probability. Since for any $\epsilon >0$, $\P( |\w S_2|>\epsilon |\mathcal F) \leq \epsilon ^{-2} \E[\w S_2^2 |\mathcal F]$, it remains to note that the sequence $\P(|\w S_2|>\epsilon|\mathcal F)$ is uniformly integrable to apply the Lebesgue domination theorem to get
\begin{align*}
\P(\w S_2>\epsilon)\overset{}{\lr}0.
\end{align*}
To conclude, we apply the CLT to $\w S_1$ and the statement follows.
\qed

\subsection{Somme lemmas}

\begin{lemma}\label{bochner} For any function $g :\R^d \r \R$, we define
\begin{align*} 
 g_h(x) = \int g(x+hu)K(u)du.
\end{align*}
Under Assumptions (B\ref{ash1}) and (B\ref{ash2})  we have
for some constant  $C$
\begin{align}\label{bochner2}
\|f_h - f\|_\infty \le C h^{r},
\end{align}
and for any $\psi\in\EuScript H_{s}$,  $\lfloor{s}\rfloor\le r$ (cf Equation~(\ref{nicol})
and the following remark)
\begin{align}\label{bochner1}
\left\| \sum_{i=1}^n \psi(X_i) -\psi_{h}(X_i)\right\|_2 \le C  n^{1/2}( h^{{s}} + n^{1/2} h^{r})
\end{align}
where $C$ depends on $\psi$ and $f$.
\end{lemma}

\begin{proof} We split mean and variance:
\begin{align*}
\E[( \sum_{i=1}^n \psi(X_i) -\psi_{h}(X_i))^2]
=&(n\E[\psi(X_1) -\psi_{h}(X_1)])^2+nVar(\psi(X_1) -\psi_{h}(X_1)).
\end{align*}
For the mean:
\begin{align*}
\E[\psi(X_1) -\psi_{h}(X_1)]
&=\int\left(\psi(x)-\psi_h(x)\right)f(x) dx\\
&=\int\psi(x)f(x)-\psi(x) f_h(x) dx\\
&=\int\psi(x)(f(x)-f_h(x)) dx \\
|\E[\psi(X_1) -\psi_{h}(X_1)]|&\le Ch^r\|\psi\|_\infty
\end{align*}
and for the variance
\begin{align}\label{dfg}
\E[(\psi_h(X_1) -\psi(X_1))^2]
=& \int\left(\int\left(\psi(x+hu)-\psi(x)\right)K(u)du \right)^2f(x)  dx.
\end{align}
By the Taylor formula with Lagrange remainder applied to $g(t)=\psi(x+tu)$ with 
$k=\lfloor{s}\rfloor$:
\begin{align*}
\psi(x+hu)=&\sum_{j=0}^{k-1}\frac{h^j}{j!}g^{(j)}(0)+
\int_0^{h}g^{(k)}(t)\frac{(h-t)^{k-1}}{(n-1)!}dt\\
=&\sum_{j=0}^k\frac{h^j}{j!}g^{(j)}(0)+
\int_0^{h}(g^{(k)}(t)-g^{(k)}(0))\frac{(h-t)^{k-1}}{(n-1)!}dt.
\end{align*}
The first term is $\psi(x)$ plus a polynomial in $u$ which will vanish
after insertion in (\ref{dfg}) because $K$ is orthogonal the first non-constant polynomial of degree
$\le r$. The second term is bounded as
 \begin{align*}
|\int_0^{h}(g^{(k)}(t)-g^{(k)}(0))\frac{(h-t)^{k-1}}{(k-1)!}dt|
\le C|u|^kh^{k-1}\int_0^{h}\|\psi^{(k)}(x+tu)-\psi^{(k)}(x)\|dt.
\end{align*}
Hence
\begin{align}\label{taylor}
|\int\left(\psi(x+hu)-\psi(x)\right)K(u)du|
\le Ch^{k-1}\int_0^{h}\int\|\psi^{(k)}(x+tu)-\psi^{(k)}(x)\||u|^k K(u)du\,dt
\end{align}
and by the generalized Minkowski inequality \cite{tsybakov2009}\footnote
{For any nonegative function $g(.,.)$ on $\mathbb R^{k+p}$,
\begin{align*}
\left(\int\left(\int g(y,x)dy\right)^2 dx\right)^{1/2}\le& \int\left(\int g(y,x)^2dx\right)^{1/2} dy
\end{align*}
}
\begin{align*}
\|\psi_h(X_1) -\psi(X_1)\|_2
\le& Ch^{k-1}\int\left(\int\|\psi^{(k)}(x+tu)-\psi^{(k)}(x)\|^2u^{2k}K(u)^2 1_{0\le t\le h}f(x)dx\right)^{1/2}  dudt\\
\le & C'h^{k-1}\int\left(|tu|^{2\alpha} |u|^{2k}K(u)^2 \right)^{1/2} 1_{0\le t\le h} dudt\\
\le & C'h^{k+\alpha}.
\end{align*}
This proves (\ref{bochner1}). Concerning (\ref{bochner2}), we use (\ref{taylor}) with $f$ and $k=r$:
\begin{align*}
|f_h(x)-f(x)|
\le& Ch^{r-1}\int_0^{h}\int\|f^{(r)}(x+tu)\||u|^r K(u)du\,dt\\
\le& C''h^{r} \int|u|^{r} K(u)du\,ds
\end{align*}

\end{proof}

\begin{theorem}\label{efronstein}(Efron-Stein inequality)
Let $X_1,\dots X_n$ be an i.i.d. sequence, $X_1'$ be an independent copy of $X_1$ 
and $f$ be a symmetric function of $n$ variables, then
\begin{align*}
Var(f(X_1,\dots X_n))\le \frac n 2\E[(f(X_1,\dots X_n)-f(X'_1,X_2,\dots X_n))^2].
\end{align*}
\end{theorem}

\begin{theorem}\label{holderremark}
If the support of $\varphi$ is a bounded convex set and 
$\varphi$ is ${\alpha}$-H\"older inside its support then $\varphi\in\EuScript H_{\min(\alpha,1/2)}$.
\end{theorem}
\begin{proof} We have
\begin{align*}
\int |\varphi(x+u)-\varphi(x)|^2 dx 
&\le \|\varphi\|_{\infty}^2 \int (\mathds 1 _ {\{x+u\in Q\}}\mathds 1_{\{ x\notin Q\}}+\mathds 1 _ {\{x+u\notin Q\}}\mathds 1_{\{ x\in Q\}}) dx +C'|u|^{2\alpha} \\ 
&\le \|\varphi\|_{\infty}^2\lambda ( y\, :\,  \emph{dist} (y,\partial Q)< |u| )+C'|u|^{2\alpha } \\
&\le \|\varphi\|_{\infty}^2 \xi_{n-1}(S) |u|+C'|u|^{2\alpha} ,
\end{align*}
where $\xi_{n-1}(S)$ is called a Quermassintegrale of Minkowski. The last inequality follows from the Steiner's formula stated for instance in \cite{federer1969}, Theorem 3.2.35 page 271.
\end{proof}

The following lemma gives some bounds on the conditional moments of $\xi_{12}$ that are useful in the proof of Theorem \ref{thelemma}.
\begin{lemma}\label{moments}
Let $\xi_{ij}=K_{ij} - f_h(X_i)$, under (B\ref{ash1}) and (B\ref{ash2})
\begin{align}
&\E[\xi_{12}|X_1]=0\label{eq1}\\
&\E[|\xi_{12}|^p|X_1]\le C h^{-(p-1)d}\label{eq2}\\
&|\E[\xi_{13}\xi_{23}|X_1,X_2]| \le \|f\|_\infty (h^{-d}K_2(h^{-1}(X_2-X_1))+\|f\|_{\infty})\label{eq3},
\end{align}
with $K_2(x)=\int|K(x-y)K(y)|dy$.
\end{lemma} 
\begin{proof}
The first equation is trivial. For the second equation, the triangular inequality and the Jensen inequality provide
\begin{align*}
\E[|\xi_{12}|^p|X_1]^{1/p}\leq 2 \E[|K_{12}|^p|X_1]=2h^{-(p-1)d}\int |K(u)|^pf(X_1+h u) d x ,
\end{align*}
and the third one is derived by 
\begin{align*}
|\E[\xi_{13}\xi_{23}|X_1,X_2]|=&|\E[\xi_{13}K_{23}|X_1,X_2]| \\=&h^{-d}|\int (h^{-d}K(h^{-1}(x-X_1))-f_h(X_1))K(h^{-1}(x-X_2))f(x)dx| \\
=&|\int (h^{-d}K(h^{-1}(X_2-X_1)+u)-f_h(X_1))K(u)f(X_2+hu)dx| \\
\le&\|f\|_\infty (h^{-d}K_2(h^{-1}(X_2-X_1))+\|f\|_{\infty}).
\end{align*}
\end{proof}

\paragraph{Acknowledgement.} The authors would like to thank C\' eline Vial for helpful comments and advices on this article. 
  
\bibliographystyle{plain}
\bibliography{bibnp}

\end{document}